\documentclass{article}

\usepackage{geometry}
\usepackage[mathcal]{euscript}
\usepackage{hyperref}
\usepackage{amsmath,amssymb,amsthm}
\usepackage{mathtools}
\usepackage{dutchcal}
\usepackage{newtxtext,newtxmath,helvet,courier}
\usepackage{fix-cm}
\DeclareMathAlphabet{\mathpzc}{OT1}{pzc}{m}{it}

\usepackage{xcolor}
\usepackage{cases}
\newtheorem{thm}{Theorem}[section]
\newtheorem{lemma}[thm]{Lemma}

\theoremstyle{remark}
\newtheorem{rem}[thm]{Remark}

\theoremstyle{definition}

\newtheoremstyle{Claim}{}{}{\itshape}{}{\itshape\bfseries}{:}{ }{#1}
\theoremstyle{Claim}

\newcommand{\R}{\mathbb{R}}

\newcommand{\eps}{\varepsilon}

\titlepage
\title{On the problem of maximal $L^q$-regularity for viscous Hamilton-Jacobi equations}
\date{\today}

\author{Marco Cirant and Alessandro Goffi}

\begin{document}

\maketitle

\begin{abstract} For $q>2, \gamma > 1$, we prove that maximal regularity of $L^q$-type holds for periodic solutions to $-\Delta u + |Du|^\gamma = f$ in $\R^d$, under the (sharp) assumption $q > d \frac{\gamma-1}\gamma$. % The proof is based on a Bernstein method.
\end{abstract}

\noindent
{\footnotesize \textbf{AMS-Subject Classification}}. {\footnotesize 35J61, 35F21, 35B65}

%\medskip
%\textcolor{blue}{DA FARE:}
%\begin{itemize}
%\item Controllare parti in rosso
%\end{itemize}

\section{Introduction}

We address here the so-called problem of maximal $L^q$-regularity for equations of the form
\begin{equation}\label{hj}
-\Delta u(x) + |Du(x)|^\gamma=f(x)\quad\text{ in } \R^d,
\end{equation}
where $\gamma > 1$, $f : \R^d \to \R$ is $1$-periodic (i.e. $f(x+z) = f(x)$ for all $x \in \R^d$, $z \in \mathbb{Z}^d$), $d \ge 1$; that is,
\begin{equation}\label{max}\tag{M}
\begin{gathered}
\text{for all $M > 0$, there exists $K > 0$ (possibly depending on $M, \gamma, q, d$) such that} \\
-\Delta u + |Du|^\gamma=f\quad\text{ in } \R^d, \quad \|f\|_{L^q(Q)} \le M \qquad \Longrightarrow \qquad \|\Delta u\|_{L^q(Q)} + \big\| |Du|^\gamma \big\|_{L^q(Q)} \le K,
\end{gathered}
\end{equation}
$Q$ being the $d$-dimensional unit cube $(-1/2,1/2)^d$. This regularity problem has been proposed by P.-L. Lions in a series of seminars and lectures (e.g. \cite{LionsSeminar, Napoli}), where he conjectured its general validity under the assumption
\begin{equation}\label{subc}\tag{A}
q > d \frac{\gamma-1}\gamma \qquad ( \ \text{and $q > 1$} \ ).
\end{equation}
Some special cases have been addressed in these seminars, but the general problem has remained so far unsolved, to the best of our knowledge. We present here a proof of \eqref{max} + \eqref{subc}, under the sole restriction $q > 2$ (which is always realized when $\gamma > 2/(d-2)$).

\medskip

Equations of the form \eqref{hj} are prototypes of semi-linear uniformly elliptic equations with super-linear growth in the gradient, and arise for example in ergodic stochastic control \cite{ArisawaLions} and in the theory of growth of surfaces \cite{KrugSpohn}. The study of regularity of their solutions has received recently a renewed interest in the theory of Mean Field Games \cite{CCPDE, LL07}. There is a vast literature on such equations and more general quasi-linear problems. While the existence of classical (or strong) solutions has been firstly investigated (see for example \cite{AC, LUbook, Lio80, Serrin}), the attention has been later on largely focused on the existence (and uniqueness) of solutions $u \in W^{1, \gamma}(Q)$ satisfying \eqref{hj} in the weak or generalized sense (typically with Dirichlet boundary conditions). See, for example, \cite{AP93, BarPo06, BMP84, BMP92, CC95, DaGP02, FM1, GT03, Bas} and more recent works \cite{AFM15, 4Nap, Dp08, DpG16}. It has been observed that due to the super-linear nature of the problem, its (weak) solvability requires $f \in L^q$, where
\[
q \ge d \frac{\gamma-1}\gamma.
\]
Such condition has been improved in the finer scale of Lorentz-Morrey spaces, and end-point situations typically require additional smallness assumptions \cite{FM2, HMV99}. It is worth observing that many results in the literature cover the case $1 < \gamma \le 2$, that is when the gradient term has at most {\em natural growth}. General results in the full range $\gamma > 1$, based on methods from nonlinear potential theory, appeared quite recently in \cite{MP16, P14, PhucAdv}.

Roughly speaking, properties \eqref{max}+\eqref{subc} say that if $f$ belongs to a sufficiently small Lebesgue space, then solutions should enjoy much better regularity than $W^{1,\gamma}$, namely be in $W^{1, q \gamma}(Q)$ (and even in $W^{2,q}(Q)$, by standard Calder\'on-Zygmund theory). Still, additional gradient regularity is typically achieved via methods that require much stronger hypotheses on the summability of $f$, being based on the classical or weak maximum principle: viscosity theory indeed requires $f$ to be bounded \cite{IshiiLions90}, while the Aleksandrov-Bakel'man-Pucci estimate needs $f \in L^d$, as in \cite{MaSoBook}. The situation is even worse when $\gamma > 2$, as one observes that general weak solutions are just H\"older continuous \cite{DaPo15}, so one has to select $u$ in a suitable class.

\medskip

Here, we look at solutions to \eqref{hj} that can be approximated by classical ones. Therefore, we will prove \eqref{max} in the form of an a priori estimate. It is known that in such a form, \eqref{max} cannot be expected in general if
\[
 1 < q \le d \frac{\gamma-1}\gamma,
\]
as described in Remark \ref{fail}. On the other hand, P.-L. Lions indicated that \eqref{max}+\eqref{subc} can be obtained in some particular cases. First, when $\gamma = 2$, the so-called Hopf-Cole transformation $v= e^{-u}$ reduces \eqref{hj} to a linear elliptic equation, and one has the result employing (maximal) elliptic regularity and the Harnack inequality. Special cases $d=1$ and $\gamma < d / (d-1)$ can be also treated. As a final suggestion, an integral version of the Bernstein method \cite{Lions85} could be implemented to prove \eqref{max} when $q$ is close enough to $d$ (see also \cite{LL}, and \cite{BardiPerthame} for further refinements of this technique), but the full regime \eqref{subc} seems to be out of range using these sole arguments.

\medskip

The Bernstein method is the starting point of our work. It consists in shifting the attention from the equation \eqref{hj} for $u$ to the equation for a suitable function of $|Du|^2$, i.e. $w=g(|Du|^2)$; if $g$ is properly chosen, the equation for $w$ enjoys a strong degree of coercivity with respect to $w$ itself, which stems from uniform ellipticity and the coercivity of the gradient term in \eqref{hj}. By a delicate combination of these two regularising effects, it is possible to produce a crucial estimate on superlevel sets of $|Du|$, i.e.
\begin{equation}\label{ine}
\left[\int_{\{|Du| \ge k\}} \Big(|Du| - k\Big)^{\gamma q}\right]^{\frac {d-2}{d}} \le \omega\Big(\big | \{|Du| \ge k\} \big |\Big) + \int_{\{|Du| \ge k\}} \Big(|Du| - k\Big)^{\gamma q}
\end{equation}
for any $k \ge 0$, where $\omega(t) \to 0$ as $t \to 0$. 
This inequality again reflects the super-linear nature of the problem, being the exponents in the two sides unbalanced. Nevertheless, it is possible to control on $\||Du|^\gamma\|_{L^q}$ as follows: \eqref{ine} guarantees that $\int_{\{|Du| \ge k\}} \Big(|Du| - k\Big)^{\gamma q}$ is either belonging to a neighborhood of zero, or to an unbounded interval (for $k$ large enough, but independent of $\||Du|^\gamma\|_{L^q}$). By the fact that $k \mapsto \int_{\{|Du| \ge k\}} \Big(|Du| - k\Big)^{\gamma q}$ is continuous and vanishes as $k \to \infty$, the second case can be ruled out, and boundedness of  $\int_{\{|Du| \ge k\}} \Big(|Du| - k\Big)^{\gamma q}$ can be then recovered up to $k = 0$. This second key step has been inspired by an interesting argument that appeared in \cite{gmp} (see also \cite{GMP14}), where $W^{1,2}$ estimates of (powers of) $u$ are obtained arguing similarly on superlevel sets of $|u|$.

\medskip

Our result reads as follows.

\begin{thm}\label{maint} Let $f \in C^1(Q)$, $d \ge 3$, $\gamma > 1 $ and
\[
q > d \frac{\gamma-1}\gamma, \qquad q > 2 .
\]
 For all $M > 0$, there exists $K = K(M, \gamma, q, d)> 0$ such that if $u \in C^3(Q)$ is a classical solution to \eqref{hj} and 
$$\|f\|_{L^q(Q)} + \|Du\|_{L^1(Q)} \le M,$$ then $$\|\Delta u\|_{L^q(Q)} + \big\| |Du|^\gamma \big\|_{L^q(Q)} \le K.$$

\end{thm}

We stress that our approach is not perturbative, in the sense that the gradient term is not treated as a perturbation of a uniformly elliptic operator (which would be natural under the growth condition $\gamma < 2$),  nor vice-versa. It applies also to equations that have the gradient term with reversed sign (Remark \ref{gham}), and to solutions in a strong sense (Remark \ref{strong}).
As far as periodicity is concerned, it is common in applications to ergodic control and Mean Field Games. The study of \eqref{max} in cases where $u$ satisfies boundary conditions, or a local version of the estimate, will be matter of future work. We also conjecture that \eqref{max} holds in the limiting case $q = d \frac{\gamma-1}\gamma$ under an additional smallness assumption on $M$, which controls the norm of $\|f\|_q$. This would be coherent with known results on the existence of weak solutions. Nevertheless, it does not seem evident how to adapt our proof to cover this end-point case.

Finally, our technique does not apply to the parabolic counterpart of \eqref{max}. In this direction, some results based on rather different duality methods developed in \cite{cg20} will appear in a forthcoming work \cite{CGpar}.\\
\par\medskip
{\bf Acknowledgements.} The authors are members of the Gruppo Nazionale per l'Analisi Matematica, la Probabilit\`a e le loro Applicazioni (GNAMPA) of the Istituto Nazionale di Alta Matematica (INdAM). This work has been partially supported by 
the Fondazione CaRiPaRo
Project ``Nonlinear Partial Differential Equations:
Asymptotic Problems and Mean-Field Games". 
\section{Proof of the main theorem}

For the sake of brevity, we will often drop the $x$-dependance of $u, Du, ...$, and the $d$-dimensional Lebsesgue measure $d x$ under the integral sign. $(x)^+ = \max\{x,0\}$ will denote the positive part of $x$, and for any $p > 1$, $p' = p/(p-1)$.

This section is devoted to the proof of Theorem \ref{maint}, which will be based on the following lemma.

\begin{lemma}\label{mainl} There exists $\delta \in (0,1)$ (depending on $\gamma, q, d$) and $\omega: [0,+\infty) \to [0,+\infty)$ (depending on $M, \gamma, q, d$) such that
\[
\lim_{t \to 0} \ \omega(t) = 0,
\]
and for all $k \ge 1$,
\begin{equation}\label{mainest}
\left(\int_Q \left(\Big( (1 + |Du|^2)^{\frac{1+\delta}2 } - k \Big)^+\right)^{\frac{q\gamma}{1+\delta}} \right)^{\frac {d-2}{d}} \le\omega\Big(|\{ 1 + |Du|^2 > k^{\frac{2}{1+\delta}} \}|\Big)+  \int_Q \left(\Big( (1 + |Du|^2)^{\frac{1+\delta}2 } - k \Big)^+\right)^{\frac{q\gamma}{1+\delta}}.
\end{equation}
\end{lemma}

We postpone the proof of the lemma, and show first how \eqref{mainest} yields the conclusion of Theorem \ref{maint}. Setting $Y_k := \int_Q \left(\Big( (1 + |Du|^2)^{\frac{1+\delta}2 } - k \Big)^+\right)^{\frac{q\gamma}{1+\delta}}$, then \eqref{mainest} reads
\begin{equation}\label{mainesty}
Y_k^{\frac {d-2} {d}} \le Y_k + \omega\Big(|\{ 1 + |Du|^2 > k^{\frac{2}{1+\delta}} \}|\Big) \qquad \text{for all $k \ge 1$}.
\end{equation}
Note that the function $F:Z\longmapsto Z^\frac{d-2}{d} - Z$ has a unique maximizer $Z^*=\left(\frac{d-2}{d}\right)^{\frac{d}{2}}$ whose corresponding value is $F(Z^*)= F^* > 0$ (which depends on $d$ only). For any $0\leq \omega < F^*$ the equation
\[
F(Z)=\omega
\]
has two roots $0 < Z^-(\omega)<Z^*<Z^+(\omega)$. Since $\lim_{t \to 0} \ \omega(t) = 0$, pick $t^* = t^*(M, \gamma, p, d)$ such that $\omega(t) < F^*$ for all $t < t^*$. By Chebyshev's inequality,
\[
\sqrt{k^{\frac{2}{1+\delta}}-1} > \frac{\|Du\|_{L^1(Q)}}{t^*} \quad \Longrightarrow \quad |\{ 1 + |Du|^2 > k^{\frac{2}{1+\delta}} \}| < t^*,
\]
hence \eqref{mainesty} yields the alternative 
\[
\forall k > \left(\frac{\|Du\|_{L^1(Q)}}{t^*} + 1\right)^{\frac{1+\delta}{2}} =: k^*, \quad Y_k < Z^* \text{\ or  \ }Y_k >  Z^*.
\]
Since $u \in C^3(Q)$, the function $k\longmapsto Y_k$ is continuous and tends to zero as $k\to\infty$ (it eventually vanishes for $k$ large). Hence we deduce that
\[
\forall k > k^*, \quad Y_k < Z^*,
\]
and finally
\[
\big\| |Du|^\gamma \big\|_{L^q(Q)}^{\frac{1+\delta}\gamma} =  \big\| |Du|^{1+\delta} \big\|_{L^{\frac{\gamma q}{1+\delta}}(Q)} \le
\big\| \Big((1+|Du|^2)^{\frac{1+\delta}2} - k^*\Big)^+ + k^*\big\|_{L^{\frac{\gamma q}{1+\delta}}(Q)} \le
(Z^*)^{\frac{1+\delta}{\gamma q}} + k^*.
\]
The estimate on $\|\Delta u\|_{L^p(Q)}$ is then straightforward.
\bigskip

Having proven Theorem \ref{maint}, we now come back to the main estimate \eqref{mainest}.

%using that $(\beta+1)\frac{d}{d-2}\geq \frac{2\gamma+\delta-1}{1+\delta}+\frac{1}{1+\delta}[\gamma(p-2)+1-\delta]$

\begin{proof}[Proof of Lemma \ref{mainl}] 
Let $w(x):=g(|Du(x)|^2)$, where $g(s) =g_\delta (s)= \frac{2}{1+\delta}(1+s)^{\frac{1+\delta}{2}}$, $\delta\in(0,1)$ to be chosen later. Note that, for any $\delta\in(0,1)$, $g$ enjoys the following properties: for all $s \ge 0$,
\begin{gather}
g'(s)s^{\frac12} \le (1+s)^{\frac\delta2},\\
g'(s) + 2sg''(s) \ge \delta g'(s). \label{g2}
\end{gather}
Note also that
\[
g'(|Du(x)|^2)=(1+|Du(x)|^2)^{\frac{\delta-1}{2}}=\left(\frac{\delta+1}{2}w \right)^{\frac{\delta-1}{1+\delta}}
\]
($g, g', g''$ below will be always evaluated at $|Du(x)|^2$).

Define $w_k=(w-k)^+\in W^{1,\infty}(Q)$ and set $\Omega_k:=\{w>k\}$. We now use $\varphi= \varphi^{(j)} = -2\partial_j(g' \, \partial_j u \, w_k^\beta)$, $j = 1, \ldots, d$ and $\beta > 1$ to be chosen later
% We have
%
%\[
%\Delta w=2g'[|D^2u|^2+Du\cdot D(\Delta u)]+4g''\sum_{j}(Du\cdot Du_{x_j})^2\ .
%\]
as test functions in the Hamilton-Jacobi equation. First, integrating by parts  and substituting $w_{x_i}=2g'Du\cdot Du_{x_i}$,
%\[
%\int_{A} \partial_i u \cdot \partial_i\varphi+\int_A \frac{1}{\gamma}|Du|^\gamma\varphi=\int_A f\varphi
%\]
%and use $\varphi= over $A=\Omega_k$, recalling that $w_k=0$ on $\partial\Omega_k$. First, since $u\in C^3$, we have
%\[
%-2\int_{Q} \partial_iu \partial_i(\partial_j(g'\partial_j uw_k^\beta))=-2\int_{Q} \partial_iu \partial_j(\partial_i(g'\partial_j uw_k^\beta))\ .
%\]
%Integration by parts yields
\begin{multline*}
\sum_j \int_Q Du \cdot D\varphi = -2 \sum_{i,j} \int_{Q} \partial_iu\cdot \partial_j \big(\partial_i(g'\partial_j uw_k^\beta) \big)=2 \sum_{i,j} \int_{Q}\partial_{ij} u \, \partial_i(g'\partial_j uw_k^\beta) =\\ %-\underbrace{\int_{\partial\Omega_k}\mathrm{div}(Duw_k^\beta g')Du\cdot \nu\,dS}_{=0}\\
 4\int_{Q}g'' \sum_j(Du\cdot Du_{x_j})^2w_k^\beta+2\int_{Q}|D^2u|^2g'w_k^\beta+\beta\int_{Q}w_k^{\beta-1}Dw_{k} \cdot D w .
\end{multline*}
Moreover, again integrating by parts,
\[
-2 \sum_j \int_{Q} |Du|^\gamma\partial_j(g'\partial_j uw_k^\beta)= \gamma \sum_j \int_{Q}w_k^\beta|Du|^{\gamma-2} Du \cdot Du_{x_j} 2g' \, \partial_j u \, w_k^\beta = \gamma \int_{Q}|Du|^{\gamma-2}Du\cdot Dw \ w_k^\beta .
\]
Noting that $w_k^{\beta-1} Dw = w_k^{\beta-1} Dw_k$ on $Q$, we end up with
\begin{multline}\label{ineq1}
\beta\int_{Q}w_k^{\beta-1}|Dw_{k}|^2+\int_{Q}\Big(4g''\sum_j(Du\cdot Du_{x_j})^2+2g'|D^2u|^2\Big)w_k^\beta+\gamma \int_{Q}|Du|^{\gamma-2}Du\cdot Dw_k\ w_k^\beta\\
=-2\int_{Q}f\mathrm{div}(g'Du\, w_k^\beta) .
\end{multline}
Note also that in \eqref{ineq1} integrating on $Q$ and on $\Omega_k$ is the same, by the fact that $w_k$ vanishes on $Q \setminus \Omega_k$. We use first Cauchy-Schwarz inequality, the equation \eqref{hj} and the inequality $(a-b)^2\geq \frac{a^2}{2}-2b^2$ for every $a,b\in\R$ to get
\[
|D^2u|^2\geq\frac1d(\Delta u)^2\geq \frac{1}{2d}|Du|^{2\gamma}-\frac{2}{d}f^2\ .
\]
Moreover, again by Cauchy-Schwarz inequality (be careful about $g'' < 0$) and \eqref{g2},
\[
g'|D^2u|^2+2g''\sum_j(Du\cdot Du_{x_j})^2\geq (g' +  2|Du|^2g'') |D^2u|^2 \ge \delta g' |D^2u|^2\ .
\]
The above inequalities then yield
\[
2g'|D^2u|^2+4g''\sum_j(Du\cdot Du_{x_j})^2\geq \delta g' |D^2u|^2+\frac{\delta}{2d}|Du|^{2\gamma}g'-\frac{2\delta}{d}f^2g'. % + 2g''\sum_j(Du\cdot Du_{x_j})^2.
\]
Note that for $\gamma>1$ it holds
\[
(1+|Du|^2)^\gamma\leq 2^{\gamma-1}(1+|Du|^{2\gamma}), \quad \text{so} \quad |Du|^{2\gamma}\geq \frac{(1+|Du|^2)^\gamma}{2^{\gamma-1}}-1
\]
and hence, we are allowed to conclude
\[
\frac{\delta}{2d}|Du|^{2\gamma}g' \geq \frac{\delta}{2^\gamma d}(1+|Du|^2)^\gamma g'-\frac{\delta}{2d} g' =  \frac{\delta}{2^\gamma d}(1+|Du|^2)^{\gamma + \frac{\delta-1}2} -\frac{\delta}{2d}g' \ .
% = \frac{\delta}{2^\gamma d}(g')^{\frac{2\gamma + \delta-1}{\delta - 1}}-\frac{\delta}{2d}g'
\]
This gives, going back to \eqref{ineq1} and substituting $(1+|Du|^2)^{\frac{1}{2}}=\left(\frac{\delta+1}{2}w\right)^{\frac{1}{1+\delta}}$,
\begin{multline}\label{ineq15}
\beta\int_{\Omega_k}w_k^{\beta-1}|Dw_{k}|^2+\delta\int_{\Omega_k} g'w_k^\beta|D^2u|^2+ c_1 \int_{\Omega_k}w^{\frac{2\gamma + \delta-1}{1+\delta}} w_k^\beta \leq \\ %+ 2\int_{Q} g''\sum_j(Du\cdot Du_{x_j})^2 w_k^\beta \leq \\
\frac{\delta}{2d} \int_{\Omega_k} (1+4f^2)g ' w_k^\beta
 -2\int_{\Omega_k}f\Delta u \, g' w_k^\beta-4\int_{\Omega_k}fg''Du \cdot (D^2u \, Du ) w_k^\beta \\
 -2\beta\int_{\Omega_k} fg'Du\cdot Dw_k w_k^{\beta-1} - \gamma \int_{\Omega_k}|Du|^{\gamma-2}Du\cdot Dw_k\ w_k^\beta,
\end{multline}
where $c_1 = c_1(\delta, d, \gamma) > 0$.

\medskip

We now estimate the five terms on the right hand side of the previous inequality. The first three terms are somehow similar: using Cauchy-Schwarz inequality and that $2sg''\leq g'$, we have for some $c_2 = c_2(\delta,d) > 0$ that
\begin{multline}\label{ineq18}
\frac{\delta}{2d} \int_{\Omega_k} (1+4f^2)g ' w_k^\beta
 -2\int_{\Omega_k}f\Delta u \, g' w_k^\beta-4\int_{\Omega_k}fg''Du \cdot (D^2u \, Du ) w_k^\beta \le \\
\frac{\delta}{2d} \int_{\Omega_k} (1+4f^2)g ' w_k^\beta +  (2d+2) \int_{\Omega_k}f|D^2u|g' w_k^\beta \leq \\ \delta \int_{\Omega_k}|D^2 u|^2g'w_k^\beta+ c_2 \int_{\Omega_k}(1+f^2)w^{\frac{1-\delta}{1+\delta}}w_k^\beta .
\end{multline}
At this stage, we make some choices for the coefficients. % First, consider $\gamma>\frac{d}{d-2}$. 
Recalling that $ \frac{d}{\gamma'} < q$, we take
\begin{equation}\label{pcho}
p=\frac2d \frac{d}{\gamma'}+\frac{d-2}{d} q, \qquad \text{and} \quad \beta =\frac{1}{1+\delta}[\gamma(p-2)+1-\delta].
\end{equation}
Note that $\frac{d}{\gamma'} < p < q$. Assuming that $p > 2$ (which is always true when $\gamma>\frac{d}{d-2}$, otherwise see the remark at the end of the proof), we have  $\beta > 1$ whenever $\delta$ is close enough to zero. Moreover,
\begin{gather}
\frac{2\gamma+\delta-1}{1+\delta}=\frac{\delta-1}{1+\delta}\frac{p}{p-2}+\beta\frac{2}{p-2}, \label{cho1} \\
(\beta+1)\frac{d}{d-2} = \frac{\gamma q}{1+\delta}\label{cho2}\ .
\end{gather}
Therefore, we apply H\"older's inequality (with conjugate exponents $p/2$ and $p/(p-2)$) and Young's inequality, and then $w_k\leq w$ together with \eqref{cho1} to obtain
\begin{multline*}
c_2 \int_{\Omega_k} (1+f^2)w^{\frac{\delta-1}{1+\delta}}w_k^\beta\leq c_2 \left(\int_{\Omega_k}(1+f^2)^{\frac p2}\right)^{\frac2p}\left(\int_{\Omega_k}w^{\frac{\delta-1}{1+\delta}\frac{p}{p-2}}w_k^{\beta\frac{p}{p-2}}\right)^{1-\frac2p}\\
 \leq c_3 \int_{\Omega_k}(1+f)^p+\frac{c_1}3 \int_{\Omega_k}w^{\frac{\delta-1}{1+\delta}\frac{p}{p-2}} w_k^{\beta\frac{2}{p-2}} w_k^\beta \\\leq c_3\int_{\Omega_k}(1+f)^p+ \frac{c_1}3\int_{\Omega_k}w^{\frac{\delta-1}{1+\delta}\frac{p}{p-2}+\beta\frac{2}{p-2}}w_k^\beta\\
 \leq c_3 \int_{\Omega_k}(1+f)^p+ \frac{c_1}3 \int_{\Omega_k}w^{\frac{2\gamma+\delta-1}{1+\delta}}w_k^\beta,
\end{multline*}
where $c_3 = c_3(\delta, d, \gamma, p) > 0$. Plugging the previous inequality into \eqref{ineq18} yields
\begin{multline}\label{ineq12}
\frac{\delta}{2d} \int_{\Omega_k} (1+f^2)g ' w_k^\beta
 -2\int_{\Omega_k}f\Delta u \, g' w_k^\beta-4\int_{\Omega_k}fg''Du \cdot (D^2u \, Du ) w_k^\beta \le \\ \delta \int_{\Omega_k}|D^2 u|^2g'w_k^\beta + c_3 \int_{\Omega_k}(1+f)^p+ \frac{c_1}3 \int_{\Omega_k}w^{\frac{2\gamma+\delta-1}{1+\delta}}w_k^\beta.
\end{multline}

The fourth term in \eqref{ineq15} is a bit more delicate, we proceed as follows. Use first that $s^{\frac12}g'(s)\leq (1+s)^{\frac{\delta}{2}}$, H\"older's and Young's inequality to get
\begin{multline*}
2\beta\int_{\Omega_k}fg'Du\cdot Dw_k w_k^{\beta-1}\leq 2\beta \int_{\Omega_k}f(1+|Du|^2)^{\frac{\delta}{2}}|Dw_k| w_k^{\beta-1}\leq \\ 
 2\beta\left(\int_{\Omega_k}w_k^{\beta-1}|Dw_k|^2\right)^{\frac12}  \left(\int_{\Omega_k}f^q\right)^{\frac{1}{q}}\left(\int_{\Omega_k}(1+|Du|^2)^{\frac{\delta }{2}\frac{pq}{q-p}}\right)^{\frac{q-p}{pq}} \left(\int_{\Omega_k}w_k^{(\beta-1)\frac{p}{p-2}}\right)^{\frac{p-2}{2p}}\leq \\
\frac{\beta}{3}\int_{\Omega_k}w_k^{\beta-1}|Dw_k|^2+\frac{c_1}3\int_{\Omega_k}  w_k^{(\beta-1)\frac{p}{p-2}}  +c_4\left(\int_{\Omega_k}f^q\right)^{\frac{p}{q}}\left(\int_{\Omega_k}(1+|Du|^2)^{\frac{\delta}{2}\frac{p q}{q-p}}\right)^{\frac{q-p}{q}},
\end{multline*}
where $c_4 = c_4(\delta,d,\gamma,\beta) > 0$. Since $k \ge 1$, $w \ge 1$ on $\Omega_k$, hence, recalling also \eqref{cho1},
\[
\int_{\Omega_k}  w_k^{(\beta-1)\frac{p}{p-2}}  = \int_{\Omega_k}w_k^{\beta\frac{2}{p-2}-\frac{p}{p-2}} w_k^{\beta} \leq \int_{\Omega_k}w^{\beta\frac{2}{p-2}-\frac{p}{p-2}} w_k^{\beta} \leq
\int_{\Omega_k} w^{\beta\frac{2}{p-2}-\frac{1-\delta}{1+\delta}\frac{p}{p-2}} w_k^{\beta} = \int_{\Omega_k}w^{\frac{2\gamma+\delta-1}{1+\delta}}w_k^\beta,
\]
so
\begin{multline}\label{ineq2}
2\beta\int_{\Omega_k}fg'Du\cdot Dw_k w_k^{\beta-1}\leq 2\beta \int_{\Omega_k}f(1+|Du|^2)^{\frac{\delta}{2}}|Dw_k| w_k^{\beta-1}\leq \\ 
 \frac{\beta}{3}\int_{\Omega_k}w_k^{\beta-1}|Dw_k|^2+\frac{c_1}3 \int_{\Omega_k}w^{\frac{2\gamma+\delta-1}{1+\delta}}w_k^\beta   +c_4\|f\|_{L^q(Q)}^p \left(\int_{\Omega_k}(1+|Du|^2)^{\frac{\delta}{2}\frac{p q}{q-p}}\right)^{\frac{q-p}{q}}.
\end{multline}

We now focus on the fifth term in \eqref{ineq15}. By Young's inequality,
\begin{equation}\label{ineq3}
- \gamma \int_{\Omega_k}|Du|^{\gamma-2}Du\cdot Dw_k\ w_k^\beta \le \frac{3 \gamma^2}{4 \beta} \int_{\Omega_k}|Du|^{2\gamma-2}\ w_k^{\beta+1} + \frac{\beta}3 \int_{\Omega_k}|Dw_k|^2 \ w_k^{\beta-1}.
\end{equation}
Furthermore, letting
\[
\eta=\frac{2\gamma+\delta-1}{1+\delta}, \qquad \text{\Big(\, so that $\beta + \eta = \frac{p\gamma}{1+\delta}$\, \Big)}
\]
we get (it holds $s^{\frac12} \le g^{\frac1{1+\delta}}$)
\[
 \int_{\Omega_k}|Du|^{2\gamma-2}\ w_k^{\beta+1} \le \int_{\Omega_k}w^{\frac{2\gamma-2}{1+\delta}}w_{k}^{\beta+1}=\int_{\Omega_k}w^{\eta-1}w_{k}^{\beta/\eta'}w_k^{\beta/\eta+1}\leq \left(\int_{\Omega_k} w^\eta w_k^\beta\right)^{\frac{1}{\eta'}}\left(\int_{\Omega_k} w_k^{\beta+\eta}\right)^{\frac1\eta}.
\]
Plugging the previous inequality into \eqref{ineq3} and using again  Young's inequality leads to
\begin{equation}\label{ineq5}
- \gamma \int_{\Omega_k}|Du|^{\gamma-2}Du\cdot Dw_k\ w_k^\beta \le \frac{c_1}3 \int_{\Omega_k} w^{\frac{2\gamma+\delta-1}{1+\delta}} w_k^\beta + c_5\int_{\Omega_k} w_k^{\frac{p\gamma}{1+\delta}}  + \frac{\beta}3 \int_{\Omega_k}|Dw_k|^2 \ w_k^{\beta-1}.
\end{equation}
for some $c_5 = c_5(\delta, d, \gamma, p) > 0$.

\medskip

Plug now \eqref{ineq12}, \eqref{ineq2} and \eqref{ineq5} into \eqref{ineq15} to obtain
\begin{equation}\label{ineq6}
\frac\beta3\int_{\Omega_k}w_k^{\beta-1}|Dw_{k}|^2 \leq 
c_3 \int_{\Omega_k}(1+f)^p  +c_4\|f\|_{L^q(Q)}^p \left(\int_{\Omega_k}(1+|Du|^2)^{\frac{\delta}{2}\frac{p q}{q-p}}\right)^{\frac{q-p}{q}} 
+  c_5\int_{\Omega_k} w_k^{\frac{p\gamma}{1+\delta}}.
\end{equation}
Sobolev's inequality related to the continuous embedding of $W^{1,2}(Q)$ into $L^{\frac{2d}{d-2}}(Q)$ reads (for $c_6 = c_6 (d, \delta, \gamma, p)$)
\[
\frac{\beta}{3}\int_{Q} w_k^{\beta-1}|Dw_{k}|^2\geq c_6 \left[\left(\int_{Q}w_k^{(\beta+1)\frac{d}{d-2}}\right)^{\frac{d-2}{d}} - \int_{Q}w_k^{\beta+1}\right],
\]
hence
\begin{multline*}
c_6 \left(\int_{\Omega_k}w_k^{(\beta+1)\frac{d}{d-2}}\right)^{\frac{d-2}{d}} \leq \\ c_3 \int_{\Omega_k}(1+f)^p  +c_4\|f\|_{L^q(Q)}^p \left(\int_{\Omega_k}(1+|Du|^2)^{\frac{\delta}{2}\frac{p q}{q-p}}\right)^{\frac{q-p}{q}} 
+  c_5\int_{\Omega_k} w_k^{\frac{p\gamma}{1+\delta}} + c_6\int_{\Omega_k}w_k^{\beta+1}.
\end{multline*}
We finally choose $\delta > 0$ small enough so that $\delta\frac{p q}{q-p} < 1$. Recall that $p < q$, so using repeatedly H\"older's and Young's inequalities we obtain
\begin{align*}
c_3 \int_{\Omega_k}(1+f)^p & \le c_3 \|1+f\|_{L^q(Q)}^p |\Omega_k|^{\frac{q-p}{q}}, \\
c_4\|f\|_{L^q(Q)}^p \left(\int_{\Omega_k}(1+|Du|^2)^{\frac{\delta}{2}\frac{p q}{q-p}}\right)^{\frac{q-p}{q}} & \le c_4 \|f\|_{L^q(Q)}^p \Big\|\sqrt{1 + |Du|^2} \Big\|_{L^1(Q)}^{\delta p} |\Omega_k|^{(1- \delta\frac{p q}{q-p} )\frac{q- p}q}, \\
c_5\int_{\Omega_k} w_k^{\frac{p\gamma}{1+\delta}} & \le \frac{c_6}2\int_{\Omega_k} w_k^{\frac{q\gamma}{1+\delta}} + c_7 |\Omega_k|, \\
c_6\int_{\Omega_k}w_k^{\beta+1} & \le \frac{c_6}2 \int_{\Omega_k}w_k^{(\beta+1)\frac{d}{d-2}} + c_8 |\Omega_k|.
\end{align*}
Recalling that $(\beta+1)\frac{d}{d-2} = \frac{q\gamma}{1+\delta}$ and $\|f\|_{L^q(Q)} + \|Du\|_{L^1(Q)} \le M$, we obtain
\begin{multline*}
\left(\int_{Q}w_k^{\frac{q\gamma}{1+\delta}} \right)^{\frac{d-2}{d}} \leq \int_{Q} w_k^{\frac{q\gamma}{1+\delta}} + \\
\frac{c_3}{c_6} \|1+f\|_{L^q(Q)}^p |\Omega_k|^{\frac{q-p}{q}}  + \frac{c_4}{c_6} \|f\|_{L^q(Q)}^p \Big\|\sqrt{1 + |Du|^2} \Big\|_{L^1(Q)}^{\delta p} |\Omega_k|^{\left(1 - \delta \frac{p q}{q-p} \right)\frac{q-p}{p}} + \frac{c_7+c_8}{c_6}  |\Omega_k| \le \\
\int_{Q} w_k^{\frac{q\gamma}{1+\delta}} \quad  + \quad 
\underbrace{ \frac{c_3}{c_6} (1+M)^p |\Omega_k|^{\frac{q-p}{q}} +  \frac{c_4}{c_6} M^p (1+M)^{\delta p} |\Omega_k|^{\left(1 - \delta\frac{p q}{q-p} \right)\frac{q-p}{p}} + \frac{c_7+c_8}{c_6}  |\Omega_k| }_{=: \ \omega(|\Omega_k|)}
\end{multline*}
Replacing $w_k$ by its definition provides the assertion (up to an additional constant in front of $\omega$).

\smallskip

If the choice of $p$ in \eqref{pcho} does not satisfy $p > 2$, just pick $\tilde p$ so that $p < \tilde p < q$ and $\tilde p > 2$, and proceed in the same way.  Then, \eqref{cho2} becomes
\begin{equation}\label{chobis}
(\beta+1)\frac{d}{d-2} > \frac{\gamma q}{1+\delta}\ ,
\end{equation}
so it suffices to apply once again H\"older's and Young's inequalities to get the same assertion (with an additional term in $ \omega$).

\end{proof}

\section{Further remarks}

\begin{rem}\label{fail} \textit{General failure of \eqref{max} when $q \le d \frac{\gamma-1}\gamma$. } In the critical case $q = d \frac{\gamma-1}\gamma$ one may consider the family of functions $v_\eps$ defined as follows, for $\eps \in (0,1]$: let $\chi \in C^\infty_0\big((1,+\infty)\big)$ be a non-negative cutoff function, $\chi \equiv 1$ on $[2,+\infty)$, and $v_\eps(x) = v_\eps(|x|)$, where
\[
v_\eps(r) = c \int_r^{1/2} s^{-\frac1{\gamma-1}}\chi\Big(\frac{s}{\eps}\Big)ds, \qquad |c|^{\gamma} = - \left(d-1 - \frac1{\gamma-1}\right)c.
\]
Then, on $B_{1/2} := \{|x| < 1/2\}$,
\[
-\Delta v_\eps + |Dv_\eps|^\gamma =\frac c \eps |x|^{-\frac1{\gamma-1}}\chi'(\eps^{-1}|x|) + |c|^\gamma\big(\chi^\gamma(\eps^{-1}|x|) - \chi(\eps^{-1}|x|)\big)|x|^{-\frac\gamma{\gamma-1}}\quad=:f_\eps(x),
\]
and $v_\eps = 0$ on $\partial B_{1/2}$. Therefore, there exists $\overline M > 0$, depending on $c, d, \gamma, \chi$ only, such that 
\[
\text{$\|f_\eps\|_{L^{d \frac{\gamma-1}\gamma}(B_{1/2})} = \overline M$ \ \ for all $\eps \in (0,1/4]$, \quad but $ \big\| |Du|^\gamma \big\|_{L^{d \frac{\gamma-1}\gamma}(B_{1/2})} \to \infty $ \ \ as $\eps \to 0$. }
\]
Note that the example is meaningful only if $\gamma > \frac d{d-1}$, that is when $d \frac{\gamma-1}\gamma > 1$. Note also that though $v_\eps$ is not periodic, being smooth on $B_{1/2}$ and vanishing on $\partial B_{1/2}$, it is straightforward to produce similar examples in the periodic setting. Finally, different choices of the truncation $\chi(|x|) = \chi_\eps(|x|)$ lead to counterexamples in the regime $q < d \frac{\gamma-1}\gamma$.

Note however that existence of weak solutions to the viscous Hamilton-Jacobi equation \eqref{hj} can be obtained when $f \in L^q(Q)$ and $q = d \frac{\gamma-1}\gamma$ (at least for the Dirichlet problem), provided that $\|f\|_{L^q}$ is small, see e.g. \cite{FM1, gmp}. Therefore, we do not exclude that \eqref{max} holds even when $q = d \frac{\gamma-1}\gamma$,  under extra smallness assumptions on $\|f\|_{L^q}$.
\end{rem}

\begin{rem} \textit{$d=1,2$. } Theorem \ref{maint} is stated in dimension $d \ge 3$, but the proof for $d = 1,2$ follows identical lines. As it usually happens, the point is that in the latter case $W^{1,2}(Q)$ is continuously embedded into $L^p(Q)$ for all finite $p \ge 1$, and not only into $L^{\frac{2d}{d-2}}(Q)$. 
\end{rem}

\begin{rem}\label{strong} \textit{Less regularity of $u$. } Theorem \ref{maint} holds more in general for (strong) solutions $u \in W^{2,q} \cap W^{1, \gamma q}(Q)$ of the equation. Indeed, consider a sequence $\psi_\eps$ of standard compactly supported regularizing kernels, and observe that $u_\eps = u \star \psi_\eps$ satisfies
\[
-\Delta u_\eps + |Du_\eps|^\gamma=f \star \psi_\eps + |Du_\eps|^\gamma - |Du|^\gamma \star \psi_\eps.
\]
For $0 < \eps \le \eps_0$
\[
\|f \star \psi_\eps + |Du_\eps|^\gamma - |Du|^\gamma \star \psi_\eps\|_{L^q(Q)} + \|Du_\eps\|_{L^1(Q)} \le M+1,
\]
so applying Theorem \ref{maint} to $u_\eps$ and passing to the limit $\eps \to 0$ yields
\[
\|\Delta u\|_{L^q(Q)} + \big\| |Du|^\gamma \big\|_{L^q(Q)} \le K(M+1, \gamma, q, d ).
\]
More generally, Theorem \ref{maint} continues to hold for solutions that can be obtained as limits of smooth approximations.
\end{rem}

\begin{rem}\label{gham}\textit{More general Hamiltonians.} Theorem \ref{maint} can be easily generalized to more general equations of the form
\[
-\Delta u + H(Du)=f,
\]
where $H : \R^d \to \R$ satisfies, e.g.
\[
\Big| H(r) - c_1 |r|^\gamma \Big| \le c_2 \qquad \text{for all $r \in \R^d$,}
\]
for some $c_1, c_2 \in \R$ and $\gamma > 1$.
\end{rem}

\medskip

\begin{flushright}
\noindent \verb"cirant@math.unipd.it"\\
Dipartimento di Matematica ``Tullio Levi-Civita''\\ Universit\`a di Padova\\ Via Trieste 63, 35121 Padova (Italy)\\
\smallskip
\noindent \verb"alessandro.goffi@math.unipd.it"\\
Dipartimento di Matematica ``Tullio Levi-Civita'' \\
Universit\`a di Padova\\
via Trieste 63, 35121 Padova (Italy)
\end{flushright}

\end{document}